\documentclass[a4paper,11pt, twoside]{amsart}
\usepackage[inner=2cm,outer=2cm,top=2.5cm,headsep=1cm, footskip=1cm,bottom=3cm]{geometry}
\usepackage[english]{babel}
\usepackage{amsfonts,latexsym,rawfonts,amsmath,amssymb,amsthm,amscd}
\usepackage{enumerate}
\usepackage{stackrel}
\usepackage{hyperref}
\usepackage{graphicx}
\usepackage{mdwlist}
\usepackage{array} 
\input xy
\xyoption{all}
\usepackage{color}
\usepackage[table]{xcolor}

\hypersetup{colorlinks, citecolor=black, filecolor=black,linkcolor=black,urlcolor=black} 

\newtheorem{prop}{Proposition}[section]
\newtheorem{teo}[prop]{Theorem}
\newtheorem{lem}[prop]{Lemma}
\newtheorem{cor}[prop]{Corollary}

\theoremstyle{definition}

\newtheorem{defi}[prop]{Definition}
\newtheorem{example}[prop]{Example}

\newcommand{\Ho}{\mathrm{Ho}}
\newcommand{\Hom}{\mathrm{Hom}}
\newcommand{\Ker}{\mathrm{Ker }}
\newcommand{\Coker}{\mathrm{Coker }}
\newcommand{\Img}{\mathrm{Im }}

\newcommand{\pb}{\ar@{}[dr]|{\mbox{\LARGE{$\lrcorner$}}}}
\newcommand{\xra}[1]{\xrightarrow{#1}}
\newcommand{\lra}{\longrightarrow}

\newcommand{\cdga}[1]{\mathsf{CDGA}_{#1}}

\newcommand{\du}{*}

\newcommand{\TW}{\text{\tiny{$\mathrm{TW}$}}}

\newcommand{\wt}{\widetilde}
\newcommand{\ov}[1]{\overline{#1}}

\newcommand{\CC}{\mathbb{C}}

\newcommand{\PP}{\mathbb{P}}
\newcommand{\QQ}{\mathbb{Q}}

\newcommand{\ZZ}{\mathbb{Z}}

\newcommand{\Aa}{\mathcal{A}}

\newcommand{\Kk}{\mathcal{K}}

\newcommand{\Mm}{\mathcal{M}}

\newcommand{\kk}{\mathbf{k}}

\newcommand{\gris}{\cellcolor{black!8}}

\title[Rational homotopy of projective varieties with normal isolated singularities]
{Rational homotopy of complex projective varieties\\ with normal isolated singularities}

\author{David Chataur}
\address[D. Chataur]{Laboratoire Ami\'{e}nois de Math\'{e}matique Fondamentale et Appliqu\'{e}e\\  Universit\'{e} de Picardie Jules Verne\\ 
33, rue Saint-Leu 80039 Amiens Cedex 1, France}
\email{david.chataur@u-picardie.fr}

\author{Joana Cirici}
\address[J. Cirici]{
Fachbereich Mathematik und Informatik\\
Freie Universit\"{a}t Berlin\\  Arnimallee 3\\ 
14195 Berlin, Germany}
\email{jcirici@math.fu-berlin.de}
\thanks{J. Cirici acknowledges financial support from the Excellence Initiative of the DFG and the SPP 1786.
Partial support from the Spanish Ministry of Economy and Competitiveness under project MTM2013-42178-P
}

\subjclass[2010]{55P62, 32S35.}
\keywords{Rational homotopy, mixed Hodge theory, weight filtration,
fundamental groups, formality, singular varieties, isolated singularities, contractions.}

\begin{document}
\maketitle

\begin{abstract}
Let $X$ be a complex projective variety of dimension $n$ with only isolated normal singularities.
In this paper, we prove, using mixed Hodge theory, that if the link of each singular point of $X$ is $(n-2)$-connected, then
$X$ is a formal topological space. 
This result applies to a large class of examples, such as normal surface singularities, varieties
with ordinary multiple points, hypersurfaces with isolated singularities
and, more generally, complete intersections with isolated singularities.
We obtain analogous results for contractions of subvarieties.
\end{abstract}

\section{Introduction}
The rational homotopy type of a topological space $X$
is the commutative differential graded algebra (cdga for short) $\Aa_{pl}(X)$ in the homotopy category $\Ho(\cdga{\QQ})$
defined by inverting quasi-isomorphisms, where
$\Aa_{pl}:\mathsf{Top}\to \cdga{\QQ}$ is Sullivan's functor of rational piece-wise linear forms.
A topological space $X$ is said to be \textit{formal} if its rational algebra  of piece-wise linear forms is a \textit{formal cdga}:
there is a string of quasi-isomorphisms from $\Aa_{pl}(X)$ to its cohomology $H^*(\Aa_{pl}(X))\cong H^*(X;\QQ)$ considered as a cdga with trivial differential.
In particular, if $X$ is formal then its rational homotopy type is completely determined by its cohomology ring, and higher order Massey products
vanish. 
Using Hodge theory, Deligne,
Griffiths, Morgan and Sullivan \cite{DGMS}
proved that smooth projective varieties or, more generally, compact K\"{a}hler manifolds, are formal. 

Simpson \cite{Si4} showed that every finitely presented group $G$
is the fundamental group $\pi_1(X)$ of an irreducible projective variety $X$.
Later, Kapovich and Koll\'{a}r \cite{KK} showed that $X$ can be chosen 
to be a complex projective surface with simple normal crossing singularities only.
This implies the existence of non-formal complex projective varieties.
For instance, one may take $G$ to be the fundamental group of the complement of the Borromean rings,
which has non-trivial triple Massey products.
In this paper, we use mixed Hodge theory to show
that a large class of projective varieties with normal isolated singularities
are formal topological spaces.
We also study the mixed Hodge structures on the rational homotopy type of a contraction of a subvariety and prove
analogous results in this setting.

We next outline the contents of this paper.
For a complex algebraic variety $X$, the complex homotopy type $\Aa_{pl}(X)\otimes\CC$
can be computed from the differential bigraded algebra defined by the first term $E_1^{*,*}(X)$ of the \textit{multiplicative weight spectral sequence},
a multiplicative analogue of Deligne's weight spectral sequence \cite{DeHIII}.
This result was proven by Morgan \cite{Mo} for smooth quasi-projective varieties and by Cirici-Guill\'{e}n \cite{CG1} in the possibly singular case.
In Section $\ref{Sweight}$, we study the multiplicative weight spectral sequence
of a projective variety with isolated singularities and provide a simple description of its algebra structure in terms of a
resolution of singularities of the variety. This  gives an upper bound on the homotopy theoretic complexity of the variety.

The idea that purity of the weight filtration implies formality is part of the folklore on mixed Hodge theory and goes back to \cite{DGMS}.
In Section $\ref{Spurity}$, we give a simple proof of a refinement of this idea: we show that the purity of the weight filtration of a complex projective variety
up to a certain degree, implies formality of the variety up to the same degree.
We prove the main results of this paper in Section $\ref{Sisolated}$. 
We first show, using the multiplicative weight spectral sequence, that every normal complex projective surface is formal.
We then generalize this result to arbitrary dimensions and prove
formality for normal projective varieties
with isolated singularities whose link is $(n-2)$-connected,
where $n$ is the dimension of the variety. In particular,
complete intersections with isolated singularities are formal. Using similar techniques, we
show that if $X$ is a projective variety with normal isolated singularities admitting a resolution of singularities
with smooth exceptional divisor, then $X$ is a formal topological space.
Lastly, in Section $\ref{contractions}$ we prove analogous results for contractions of subvarieties. In particular, we show that if $Y\hookrightarrow X$
is a closed immersion of smooth projective varieties, then $X/Y$ is a formal topological space.

The results of this paper arose from a more general study (in preparation) of the vanishing of Massey products on the intersection cohomology
of projective varieties with only isolated singularities. However,
the simplicity of the statements and proofs in the classical setting of rational homotopy encouraged us to write them up separately.
Our more general study follows the intersection-homotopy treatment of \cite{CST},
where Chataur, Saralegi and Tanr\'{e} prove intersection-formality for certain spaces. In particular, they show
that any nodal hypersurface in $\CC \PP^4$ is formal.

\subsection*{Acknowledgments}
We would like to thank J. W. Morgan for asking us about 1-formality,
and A. Dimca for pointing out to us the results on the realization of fundamental groups.
Thanks also to F. Guill\'{e}n and V. Navarro-Aznar for useful comments.

J. Cirici acknowledges the Simons
Center for Geometry and Physics in Stony Brook at which part of this paper was written.

\section{Multiplicative weight spectral sequence}\label{Sweight}

Deligne proved that the $k$-th cohomology space $H^k(X;\QQ)$ of every complex algebraic variety $X$ 
carries a functorial \textit{mixed Hodge structure}:
this is given by an increasing filtration 
$$0=W_{-1}H^k(X;\QQ)\subset W_0H^k(X;\QQ)\subset\cdots\subset W_{2k}H^k(X;\QQ)=H^k(X;\QQ)$$
of the rational cohomology of $X$, called the \textit{weight filtration},
together with a decreasing filtration $F$
of the complex cohomology $H^k(X;\CC)$, called the \textit{Hodge filtration}, 
such that $F$
and its complex conjugate induce a Hodge decomposition of weight $p$
on each graded piece $Gr_p^WH^k(X;\CC)$.
If $X$ is smooth then $W_{k-1}H^k(X;\QQ)=0$, while if  $X$ is projective then $W_kH^k(X;\QQ)=H^k(X;\QQ)$.
Let us briefly explain how $W$ is defined in the latter case.
Let $X_\bullet\to X$ be a simplicial resolution of $X$:
this is a smooth simplicial variety  $X_\bullet$ together with an augmentation morphism $X_\bullet\to X$ 
satisfying cohomological descent (see \cite{DeHIII}, see also \cite{PS}, II.5).
Deligne showed that the associated spectral sequence
$\wt E_1^{p,q}(X):=H^q(X_p;\QQ)\Rightarrow H^{p+q}(X;\QQ)$
degenerates at the second stage, and that the induced filtration on the rational cohomology of $X$
is well-defined (does not depend on the chosen resolution) and is functorial for morphisms of varieties.
The \textit{weight filtration $W$} on $H^*(X;\QQ)$ is then defined
by d\'{e}calage of the induced filtration. We have
$Gr_q^WH^{p+q}(X;\QQ)\cong \wt E_2^{p,q}(X).$

Let $k\geq 0$. The weight filtration on $H^k(X;\QQ)$ is said to be \textit{pure of weight $k$} if
$$0=W_{k-1}H^k(X;\QQ)\subset W_{k}H^k(X;\QQ)=H^k(X;\QQ).$$
For instance, if $X$ is smooth and projective then $W$ on $H^k(X;\QQ)$ is pure of weight $k$, for all $k\geq 0$.
A well-known consequence of the Decomposition Theorem of
intersection homology is that a
projective variety $X$ with only isolated singularities satisfies \textit{semi-purity}: 
the weight filtration on $H^k(X;\QQ)$ is pure of weight $k$,
for all $k>n$, where $n=\dim(X)$ (\cite{Ste}, see also \cite{Na2} for a more direct proof).

Morgan \cite{Mo} introduced mixed Hodge diagrams of differential
graded algebras and proved the existence of functorial mixed Hodge structures on the
rational homotopy groups of smooth complex algebraic varieties.
His results were extended to the singular setting by Hain \cite{HaBull} and Navarro-Aznar \cite{Na} independently.
In particular, Navarro-Aznar introduced the Thom-Whitney simple functor
and developed the construction of algebras of piece-wise linear forms associated to simplicial varieties.
As a result, for any complex algebraic variety $X$, one has a \textit{multiplicative weight spectral sequence} $E_1^{*,*}(X)$ which 
is a well-defined algebraic invariant of $X$ in the homotopy category of differential bigraded algebras, and is homotopy equivalent, when 
forgetting the algebra structure, to Deligne's weight spectral sequence $\wt E_1^{*,*}(X)$. 
Furthermore, the multiplicative weight spectral sequence 
carries information on the rational homotopy type of the variety.

\subsection{Mixed Hodge diagrams and multiplicative weight spectral sequence}We next recall the definition of the multiplicative weight 
spectral sequence associated with a complex algebraic variety.

\begin{defi}
A \textit{filtered cdga} $(A,W)$ over a field $\kk$ is a cdga $A$ over $\kk$ together with an (increasing) filtration
$\{W_pA\}$ indexed by $\ZZ$ such that $W_{p-1}A\subset W_pA$, $d(W_pA)\subset W_pA$, and $W_pA\cdot W_qA\subset W_{p+q}A$.
\end{defi}

Every filtered cdga $(A,W)$ has an associated spectral sequence, each of whose stages $(E^{*,*}_r(A,W),d_r)$ is
a differential bigraded algebra with differential of bidegree $(r,1-r)$.

\begin{defi}
A \textit{mixed Hodge diagram (of cdga's over $\QQ$)} consists of
a filtered cdga $(\Aa_{\QQ},W)$ over $\QQ$, a bifiltered cdga $(\Aa_{\CC},W,F)$ over $\CC$, together with
a string of filtered quasi-isomorphisms
from $(\Aa_\QQ,W)\otimes \CC$ to $(\Aa_\CC,W)$.
In addition, the following axioms are satisfied:
\begin{enumerate}
\item[($\mathrm{MH}_0$)] The weight filtration $W$ is regular and exhaustive. The Hodge filtration $F$ is biregular.
The cohomology $H^*(\Aa_\QQ)$ has finite type.
\item[($\mathrm{MH}_1$)] For all $p\in\ZZ$, the differential of $Gr_p^W\Aa_\CC$ is strictly compatible with $F$.
\item[($\mathrm{MH}_2$)] For all $k\geq 0$ and all $p\in\ZZ$, the filtration $F$ induced on $H^k(Gr^W_p\Aa_{\CC})$ defines a pure Hodge structure of
weight $p+k$ on $H^k(Gr^W_p\Aa_\QQ)$.
\end{enumerate}
Such a diagram is denoted as $\Aa=\left((\Aa_\QQ,W)\stackrel{\varphi}{\dashleftarrow\dashrightarrow}(\Aa_\CC,W,F)\right).$
\end{defi}

By forgetting the multiplicative structures we recover the original notion of mixed Hodge complex introduced by Deligne
(see 8.1 of \cite{DeHIII}). In particular, the $k$-th cohomology group of every mixed Hodge diagram
is a mixed Hodge structure.

The following is a multiplicative version of Deligne's Theorem \cite{DeHIII} on the existence of functorial mixed Hodge structures in cohomology.
\begin{teo}[$\S$9 of \cite{Na}, see also \cite{HaBull}]\label{extensionavarro}
For every complex algebraic variety $X$ there exists a mixed Hodge diagram $\Aa(X)$ 
such that $\Aa(X)_\QQ\simeq \Aa_{pl}(X)$ and for all $k\geq 0$, the cohomology $H^k(\Aa(X))$ 
is isomorphic to Deligne's mixed Hodge structure on $H^k(X;\QQ)$.
This construction is well-defined and functorial for morphisms of varieties in the homotopy category of mixed Hodge diagrams.
\end{teo}

The rational component $(\Aa_\QQ,W)$ of every mixed Hodge diagram $\Aa$ is a filtered cdga over the rationals. Hence it 
has an associated spectral sequence $E_1^{*,*}(\Aa_\QQ,W)\Rightarrow H^*(X;\QQ)$.

\begin{teo}[\cite{CG1}, Theorem 3.23]\label{E1formal}
Let $\Aa$ be mixed Hodge diagram such that $H^0(\Aa_\QQ)\cong \QQ$. There is a 
string of quasi-isomorphisms of complex cdga's 
from $(\Aa_\QQ,W)\otimes\CC$ to $E_1(\Aa_\QQ,W)\otimes\CC$ compatible with the filtration $W$.
\end{teo}
The proof of the above result uses minimal models in the sense of rational homotopy. 
This is why we ask that mixed Hodge diagrams are cohomologically connected.
Under some finite type conditions the above result is also valid over $\QQ$ (see \cite{CG1}, Theorem 2.26).
We remark that in \cite{CG1}, a weaker notion of mixed Hodge diagram is used. This does not affect the above result.
\begin{defi}
Let $X$ be a complex algebraic variety. The \textit{multiplicative weight spectral sequence of $X$} is the spectral sequence
$$E_1^{*,*}(X):=E_1^{*,*}(\Aa(X)_\QQ,W)\Longrightarrow H^*(X;\QQ)$$
 associated with the filtered cdga $(\Aa(X)_\QQ,W)$ given by the rational component of $\Aa(X)$.
\end{defi}
This is a well-defined algebraic invariant of $X$ in the homotopy category of differential bigraded algebras.
When forgetting the multiplicative structures, there is a homotopy equivalence of bigraded complexes between $E^{*,*}_1(X)$ and
Deligne's weight spectral sequence $\wt E^{*,*}_1(X)$.

The main advantage of the multiplicative weight spectral sequence with respect to Deligne's weight spectral sequence
is that by Theorem $\ref{E1formal}$, the former carries information about the rational homotopy type of $X$.

\subsection{Thom-Whitney simple}Theorem $\ref{extensionavarro}$ relies on the Thom-Whitney simple functor, which
associates, to every strict cosimplicial cdga $A^\bullet$ over a field $\kk$ of characteristic zero,
a new cdga $s_\TW(A^\bullet)$ over $\kk$. This  cdga is homotopically equivalent, as a complex, 
to the total complex  of the original cosimplicial cdga.
Hence the Thom-Whitney 
simple can be viewed as a multiplicative version of the total complex. We next recall its construction.
For every $\alpha\geq 0$, denote by $\Omega_\alpha$ the cdga given by
$$\Omega_\alpha:={{\Lambda(t_0,\cdots,t_\alpha,dt_0,\cdots,dt_\alpha)}\over{\sum t_i-1,\sum dt_i}}$$
where $\Lambda_\alpha:=\Lambda(t_0,\cdots,t_\alpha,dt_0,\cdots,dt_\alpha)$ denotes the free cdga over $\kk$
generated by $t_i$ in degree 0 and $dt_i$ in degree 1. 
The differential 
on $\Lambda_\alpha$ is defined by $d(t_i)=dt_i$ and $d(dt_i)=0$. 
For all $0\leq i\leq \alpha$, define face maps $\delta_\alpha^i:\Omega_\alpha\to \Omega_{\alpha-1}$  by letting
$$
\delta_\alpha^i(t_k)=\left\{
\begin{array}{lll}
t_k&,& k<i\\
0&,& k=i\\
t_{k-1}&,&k>i
\end{array}
\right..
$$
These definitions make $\Omega_\bullet=\{\Omega_\alpha,\delta_\alpha^i\}$ into a strict simplicial
cdga (the adjective \textit{strict} accounts for the fact
that we do not require degeneracy maps).

Recall that the total complex $s(K^\bullet)$ of a strict cosimplicial cochain complex $K^\bullet$ is given by
$$s(K^\bullet):=\int_\alpha K^\alpha\otimes C^*(\Delta^{|\alpha|})=\bigoplus_\alpha K^\alpha [-|\alpha|],$$
where $C^*(\Delta^{|\alpha|})$ denotes the cochain complex of $\Delta^{|\alpha|}$. Analogously, we have:
\begin{defi}[\cite{Na}, $\S3$]
Let $A^\bullet$ be a strict cosimplicial cdga over $\kk$.
The \textit{Thom-Whitney simple of $A^\bullet$} is the cdga over $\kk$ defined by the end
$$s_\TW(A^\bullet):=\int_\alpha A^\alpha\otimes\Omega_\alpha.$$
\end{defi}

\begin{example}\label{doublemappingpath}
Let $f,g:A\rightrightarrows B$ be morphisms of cdga's.
Then 
$s_\TW(f,g)$ is given by the pull-back
$$
\xymatrix{
\pb\ar[d]
s_\TW(f,g)\ar[r]&B\otimes\Lambda(t,dt)\ar[d]^{(\delta^0,\delta^1)}\\
A\ar[r]^{(f,g)}&B\times B
}
$$
where $\delta^x:B\otimes\Lambda(t,dt)\to B$ is the evaluation map given by $t\mapsto x$ and $dt\mapsto 0$.
\end{example}

We shall need the following filtered version of the Thom-Whitney simple. Given $r\geq 0$,
consider on the cdga $\Omega_\alpha$ the multiplicative increasing filtration $\sigma[r]$ defined by letting
$t_i$ be of weight 0 and $dt_i$ be of weight $-r$, for all $0\leq i\leq \alpha$.
Note that $\sigma[0]$ is the trivial filtration, while $\sigma[1]$ is the b\^{e}te filtration.
\begin{defi}\label{filTW}
Let $(A,W)^\bullet$ be a strict cosimplicial filtered cdga. The \textit{$r$-Thom-Whitney simple} of $(A,W)^\bullet$ is the filtered cdga
$(s_{\TW}(A^\bullet),W(r))$ defined by
$$W(r)_ps_{\TW}(A^\bullet):=\int_\alpha\sum_q  \left(W_{p-q}A^\alpha\otimes \sigma[r]_{q}\Omega_\alpha\right).$$
\end{defi}
To study the weight spectral sequence we shall mostly be interested in the behavior of the filtration $W(1)$ of the Thom-Whitney simple,
which in the setting of complexes corresponds to the diagonal filtration $\delta(W,L)$ of $s(-)$ introduced in 7.1.6 of \cite{DeHIII}
(see also Section $\S$2 of \cite{CG3}). 
The following is a matter of verification.
\begin{lem}\label{sssimple}
Let $(A,W)^\bullet$ be a finite strict cosimplicial filtered cdga. The spectral sequence associated
with $(s_\TW(A^\bullet),W(1))$ satisfies
$$E_1^{p,q}(s_{\TW}(A^\bullet),W(1)):=H^{p+q}(Gr_{-p}^{W(1)}s_{\TW}(A^\bullet))\cong
\int_\alpha \sum_m \left(E_1^{p-m,q}(A^\alpha,W)\otimes \Omega_\alpha^m\right).
$$
\end{lem}

Deligne's simple of a cosimplicial mixed Hodge complex $\Kk^\bullet$
is the diagram of complexes given by
$$s_{D}(\Kk^\bullet):=\left((s(\Kk_\QQ^\bullet),\delta(W,L))
\stackrel{s(\varphi)}{\dashleftarrow\dashrightarrow} (s(\Kk_\CC^\bullet),\delta(W,L),F)\right),$$
which by Theorem 8.1.15 of \cite{DeHIII}, is again a mixed Hodge complex.
The following Lemma is a multiplicative analogue of this result and
gives a Thom-Whitney simple in the category of mixed Hodge diagrams.
\begin{lem}[$\S7.11$ of \cite{Na}]\label{TWisMHD}
Let $\Aa^\bullet$ be a strict cosimplicial mixed Hodge diagram.
The diagram of cdga's given by
$$s_{\TW}(\Aa^\bullet):=\left((s_{\TW}(\Aa_\QQ^\bullet),W(1))\stackrel{s_{\TW}(\varphi)}{\dashleftarrow\dashrightarrow}
(s_{\TW}(\Aa_\CC^\bullet),W(1),F(0))\right)$$
is a mixed Hodge diagram, which is homotopy equivalent, as a complex, to Deligne's simple $s_{D}(\Aa^\bullet)$.

\end{lem}

\subsection{Main examples}\label{seccioexemples}We next give a description of the multiplicative weight 
spectral sequence in some particular situations of interest for this paper.

\subsubsection*{Smooth projective varieties}
Let $X$ be a smooth projective variety. A mixed Hodge diagram $\Aa(X)$ for $X$ is given by the data
$\Aa(X)_\QQ:=\Aa_{pl}(X)$ and  $\Aa(X)_\CC:=\Aa_{dR}(X;\CC)$ the complex de Rham algebra, with $W$ the trivial
filtration and $F$ the classical Hodge filtration.
The multiplicative weight spectral sequence satisfies $E_1^{0,*}(X)=H^*(X;\QQ)$ and $E_1^{p,*}(X)=0$ for all $p>0$.

\subsubsection*{Varieties with normal crossings}
Let $D\hookrightarrow \wt X$ be a simple normal crossings divisor in a smooth projective variety $\wt X$ of dimension $n$. We may write
$D=D_1\cup\cdots \cup D_N$ as the union of
of irreducible smooth varieties meeting transversally. Let $D^{(0)}=\wt X$ and for all $p>0$, denote by $D^{(p)}=\bigsqcup_{|I|=p}D_I$
the disjoint union of all 
$p$-fold intersections 
$D_I:=D_{i_1}\cap\cdots \cap D_{i_p}$ where $I=\{i_1,\cdots,i_p\}$ denotes an ordered subset of $\{1,\cdots,N\}$.
Since $D$ has normal crossings, it follows that $D^{(p)}$ is a smooth projective variety of dimension $n-p$. 
For $1\leq k\leq p$, denote by $j_{I,k}:D_I\hookrightarrow D_{I\setminus \{i_k\}}$ the inclusion 
and let $j_{p,k}:=\bigoplus_{|I|=p} j_{I,k}:D^{(p)}\hookrightarrow D^{(p-1)}$.
This defines a simplicial resolution
$D_\bullet=\{D^{(p)}, j_{p,k}\}\lra D$, called the \textit{canonical hyperresolution} of $D$.
Deligne's weight spectral sequence is the first quadrant spectral sequence given by (see for example $\S4$ of \cite{GS})
$$\wt E_1^{p,q}(D)=H^q(D^{(p+1)};\QQ)\Longrightarrow H^{q+p}(D;\QQ)$$
with the differential $\wt d_1^{p,*}=j_{p+2}^*$ defined via the 
combinatorial restriction morphisms 
$$j_{p}^*:=\sum_{k=1}^p (-1)^{k-1} (j_{p,k})^{*}:H^{*}(D^{(p-1)};\QQ)\lra H^*(D^{(p)};\QQ).$$
By computing the cohomology of $\wt E^{*,*}_1(D)$ we obtain:
$$Gr^W_q H^{p+q}(D;\QQ)\cong \wt E_2^{p,q}(D)=\Ker (d_1^{p,q})/\Img(d_1^{p-1,q})\cong
\left\{ 
\begin{array}{lll}
 \Ker(j_{2}^*)^q&,&p=0\\
  \Ker(j_{p+2}^*)^q/\Img(j_{p+1}^*)^q&,&p>0
\end{array}\right..
$$

\begin{prop}\label{ssdivisor}
Let $D_\bullet=\{D^{(p)}, j_{p,k}\}\lra D$ be the canonical hyperresolution of a simple normal crossings divisor $D$.
The multiplicative weight spectral sequence of $D$ is given by
 $$E_1^{p,q}(D)\cong \int_\alpha  H^{q}(D^{(\alpha+1)};\QQ)\otimes \Omega_\alpha^p\Longrightarrow H^{p+q}(D;\QQ).$$
\end{prop}
\begin{proof}
Since $D^{(p)}$ is smooth and projective for each $p$, there is a strict cosimplicial 
mixed Hodge diagram $\Aa(D_\bullet):=\{\Aa(D^{(p)}),j_{p,k}^*\}$ with trivial weight filtrations. A
mixed Hodge diagram for $D$ is then given by the Thom-Whitney simple
$\Aa(D):=s_{\TW}(\Aa(D_\bullet))$. Indeed,
cohomological descent for rational homotopy gives a quasi-isomorphism $\Aa(D)_\QQ\simeq \Aa_{pl}(X)$. 
Furthermore, the filtrations of $\Aa(D)$ given in Lemma $\ref{TWisMHD}$ make $\Aa(D)$ into a mixed Hodge diagram.
By Lemma $\ref{TWisMHD}$, when forgetting the multiplicative structures, $\Aa(D)$ is quasi-isomorphic to the canonical mixed Hodge complex for $D$ 
(see also $\S4$ of \cite{GS}). Hence it induces Deligne's mixed Hodge structure on the cohomology of $D$.
The result follows from Lemma
$\ref{sssimple}$, since $E_1^{0,q}(D^{(\alpha)})=H^q(D^{(\alpha)};\QQ)$ and $E_1^{p,q}(D^{(\alpha)})=0$ for $p>0$.
\end{proof}

\begin{example}Let $D$ be a simple normal crossings divisor of complex dimension $1$. Then 
$D^{(1)}$ is a disjoint union of smooth projective curves, $D^{(2)}$ is a collection of points and
$D^{(p)}=\emptyset$ for all $p>2$.
Deligne's weight spectral sequence is given by
$$
\wt E_1^{*,*}(D)\cong
\def\arraystretch{1.5}
\begin{array}{|ccc}
H^2(D^{(1)};\QQ)&&\\
H^1(D^{(1)};\QQ)&&\\
H^0(D^{(1)};\QQ)&\xra{\,\,\,j_2^*\,\,\,}& H^0(D^{(2)};\QQ)\\
\hline
\end{array}
$$
where $j_2^*:=j_{2,1}^*-j_{2,2}^*$.
The multiplicative weight spectral sequence is given by
$$
E_1^{*,*}(D)\cong
\def\arraystretch{1.5}
\begin{array}{|ccc}
H^2(D^{(1)};\QQ)&&\\
H^1(D^{(1)};\QQ)&&\\
E_1^{0,0}(D) &\xra{\,\,d_1\,\,}& H^0(D^{(2)};\QQ)\Lambda(t)dt\\
\hline
\end{array}
$$
where the term $E_1^{0,0}(D)$ is given by the pull-back
$$
\xymatrix{\pb
\ar[d]
E_1^{0,0}(D)\ar[rr]&&H^0(D^{(2)};\QQ)\otimes\Lambda(t)\ar[d]^{(\delta^0,\delta^1)}\\
H^0(D^{(1)};\QQ)\ar[rr]^-
{(j_{2,1}^*,j_{2,2}^*)}&&H^0(D^{(2)};\QQ)\times H^0(D^{(2)};\QQ)
}
$$
and the differential $d_1:E_1^{0,0}(D)\to E_1^{1,0}(D)$ is given by $(a,b(t))\mapsto b'(t)dt$.
The products $E_1^{0,0}\times E_1^{0,k}\to E_1^{0,k}$ and $E_1^{0,0}\times E_1^{1,0}\to E_1^{1,0}$
are given by $(a,b(t))\cdot c=a\cdot c$ and $(a,b(t))\cdot c(t)dt=b(t)\cdot c(t)dt$ respectively.
The unit is $(1_{D^{(1)}},1_{D^{(2)}})\in E_1^{0,0}(D)$.
The maps $H^0(D^{(1)};\QQ)\lra E_1^{0,0}$ and $H^0(D^{(2)};\QQ)\lra E_1^{1,0}$
given by $a\mapsto (a, j_{21}^*(a)(1-t)+j_{22}^*(a)t)$ and $b\mapsto -b\cdot dt$
define an inclusion $\wt E_1(D)\lra E_1(D)$ of bigraded complexes which is a quasi-isomorphism.
\end{example}

\subsubsection*{Smooth quasi-projective varieties}
Let $X$ be a smooth projective variety and let $X\hookrightarrow \wt X$ be a smooth compactification
of $X$ such that the complement $D:=\wt X-X$ is a union of divisors with normal crossings.
A mixed Hodge diagram for $X$ is defined via the 
algebra of forms on $\wt X$ which have logarithmic poles along $D$ (see $\S3$ of \cite{Mo} for details, see also $\S8$ of \cite{Na}). 
In this case, the multiplicative weight spectral sequence
coincides with Deligne's weight spectral sequence, given by
$$E_1^{-p,q}(X)=H^{q-2p}(D^{(p)};\QQ)\Longrightarrow H^{q-p}(X;\QQ).$$
The differential $d_1^{-p,*}:E_1^{-p,*}(X)\lra E_1^{-p+1,*}(X)$
is given by the combinatorial Gysin map
$$\gamma_{p}:=\sum_{k=1}^p (-1)^{k} (j_{p,k})_{!}:H^{*-2}(D^{(p)};\QQ)\lra H^*(D^{(p-1)};\QQ).$$
where $j_{p,k}:D^{(p)}\hookrightarrow D^{(p-1)}$ are the inclusion maps.
The algebra structure of $E_1^{*,*}(X)$ is induced by the combinatorial restriction morphisms
$j_{p}^*=\sum_{k=1}^p (-1)^{k-1} (j_{p,k})^{*}$ together with the cup product of $H^*(D^{(p)};\QQ)$, for $p>0$.

\subsubsection*{Isolated singularities}
Let $X$ be a complex projective variety of dimension $n$ with only isolated singularities and denote by $\Sigma$ the singular locus of $X$.
By Hironaka's Theorem on resolution of singularities
there 
exists a cartesian diagram
$$
\xymatrix{
D\ar[r]^j\ar[d]_g&\wt X\ar[d]^f\\
\Sigma\ar[r]^i&X
}
$$
where $\wt X$ is smooth, $f:\wt X\to X$ is a proper birational morphism which is an isomorphism outside
$\Sigma$ and  $D=f^{-1}(\Sigma)$ is a simple normal crossings divisor.
Since both $\wt X$ and $\Sigma$ are smooth and projective, there are mixed Hodge diagrams $\Aa(\wt X)$ and $\Aa(\Sigma)$ for $\wt X$ and $\Sigma$ respectively
with trivial weight filtration. Let $\Aa(D)=s_{\TW}(\Aa(D_\bullet))$ be a mixed Hodge diagram for $D$ as constructed in
the proof of Proposition $\ref{ssdivisor}$.
The maps $j_1:D^{(1)}\to \wt X$ and $g_1:D^{(1)}\to \Sigma$ defined by composing the map $D^{(1)}\to D$ with $j$ and $g$ respectively
induce morphisms of mixed Hodge diagrams 
$j_1^*:\Aa(\wt X)\to \Aa(D)$ and $g_1^*:\Aa(\Sigma)\to \Aa(D)$.
\begin{prop}\label{isolsingE1}
Let $X$ be a complex projective variety with only isolated singularities.
With the above notation we have:
\begin{enumerate}[(1)]
 \item A mixed Hodge diagram for $X$ is given by 
$\Aa(X)=s_{\TW}\left(\Aa(\Sigma)\times \Aa(\wt X)\rightrightarrows\Aa(D)\right).$
\item For $q\geq 0$, the term $E_1^{0,q}(X)$ is given by the pull-back
$$
\xymatrix{\pb
\ar[d]
E_1^{0,q}(X)\ar[rr]&&E_1^{0,q}(D)\otimes\Lambda(t)\ar[d]^{(\delta^0,\delta^1)}\\
H^q(\Sigma;\QQ)\times H^q(\wt X;\QQ)\ar[rr]^-
{\left(\begin{smallmatrix}
g_1^*&0\\
0&j_1^*
\end{smallmatrix}\right)
}&&E_1^{0,q}(D)\times E_1^{0,q}(D)
},
$$
while for $p>0$ and $q\geq 0$ we have:
$$E_1^{p,q}(X)\cong E_1^{p,q}(D)\otimes\Lambda(t)\oplus E_1^{p-1,q}(D)\otimes\Lambda(t)dt.$$
The differential of $E_1^{*,*}(X)$ is defined component-wise, via the differential of $E_1^{*,*}(D)$ and the differentiation with respect to $t$.
\item Denote by $j_p^*:H^*(D^{(p-1)};\QQ)\lra H^*(D^{(p)};\QQ)$ the combinatorial restriction maps, with $D^{(0)}=\wt X$.
Let $\tau:H^0(\wt X;\QQ)\times H^0(\Sigma;\QQ)\lra H^0(D^{(1)};\QQ)$ be given by
$\tau(x,\sigma)=j_1^*(x)-g_1^*(\sigma)$. 
Then
$$E_2^{p,q}(X)\cong
\def\arraystretch{1.5}
\begin{array}{| c | c | c }
\Ker(j_1^*)&\Ker(j_2^*)/\Img(j_1^*)&\Ker(j_{p+1}^*)/\Img(j_p^*)\\\hline
\Ker(\tau)&\Ker(j_2^*)/\Img(\tau)&\Ker(j_{p+1}^*)/\Img(j_p^*)\\\hline
\multicolumn{1}{c}{\text{\tiny{$p=0$}}}&\multicolumn{1}{c}{\text{\tiny{$p=1$}}}&\multicolumn{1}{c}{\text{\tiny{$p\geq 2$}}}
\end{array}
\def\arraystretch{1.5}
\begin{array}{c}
\text{\tiny{$q\geq 1$}}\\
\text{\tiny{$q=0$}}\\
\\
\end{array} $$
\end{enumerate}
\end{prop}
\begin{proof}
By Lemma $\ref{TWisMHD}$ the Thom-Whitney simple of a strict cosimplicial mixed Hodge diagram,
with the filtrations $W(1)$ and $F(0)$,
is a mixed Hodge diagram. Hence $\Aa(X)$ is a mixed Hodge diagram, which by cohomological descent satisfies
$\Aa(X)_\QQ\simeq \Aa_{pl}(X)$. By forgetting the multiplicative structures, we obtain a mixed Hodge complex for $X$
(see for example 2.9 of \cite{Durfee}).
This proves (1). Assertion (2) now follows from (1) and Lemma
$\ref{sssimple}$. Assertion (3) is a matter of verification.
\end{proof}

\section{Purity implies formality}\label{Spurity}
In this section we give a simple proof of the fact that the purity of the weight filtration of a complex projective variety
up to a certain degree, implies formality of the variety up to the same degree. A direct application is the formality of 
the Malcev completion of the fundamental group of projective varieties with normal isolated singularities.

\begin{defi}Let $r\geq 0$ be an integer.
A morphism of cdga's $f:A\to B$ is called \textit{$r$-quasi-isomorphism} if the induced morphism in cohomology $H^i(f):H^i(A)\to H^i(B)$ is an isomorphism 
for all $i\leq r$ and a monomorphism for $i=r+1$.
\end{defi}

\begin{defi}
A cdga $(A,d)$  over $\kk$ is called \textit{$r$-formal} if there is a string of $r$-quasi-isomorphisms from $(A,d)$
to its cohomology $(H^*(A;\kk),0)$ considered as a cdga with trivial differential.
We will say that a topological space $X$ is \textit{$r$-formal}
if the rational cdga $\Aa_{pl}(X)$ is $r$-formal.
\end{defi}

The case $r=1$ is of special interest, since 1-formality implies that the rational Malcev completion of
$\pi_1(X)$ can be computed directly from
the cohomology group $H^1(X;\QQ)$, together with the cup product $H^1(X;\QQ)\otimes H^1(X;\QQ)\to H^2(X;\QQ)$.
In this case we say that \textit{$\pi_1(X)$ is formal}.
For $r=\infty$ we recover the usual notion of formality, which in the case of simply connected (or more generally, nilpotent) spaces, implies 
that the higher rational homotopy groups $\pi_i(X)\otimes\QQ$, with $i>1$ can be computed directly from the cohomology ring $H^*(X;\QQ)$.
Note that if $X$ is formal, then $\pi_1(X)$ is also formal.

\begin{teo}\label{purityformality}
Let $X$ be a complex projective variety and let $r\geq 0$ be an integer.
If the weight filtration on $H^k(X;\QQ)$ is pure of weight $k$, for all $0\leq k\leq r$, then $X$ is $r$-formal.
\end{teo}
\begin{proof}
We prove formality over $\CC$ and apply independence of formality on the base field for cdga's with 
cohomology of finite type (see \cite{Su}, see also \cite{HS}).
Since the disjoint union of $r$-formal spaces is $r$-formal, we may assume that $X$ is
connected, so it has a mixed Hodge diagram $\Aa(X)$ with $H^0(\Aa(X)_\QQ)\cong \QQ$.
By Theorem $\ref{E1formal}$ it suffices to define a string of $r$-quasi-isomorphisms 
of differential bigraded algebras from $(E_1^{*,*}(X),d_1)$ to $(E_2^{*,*}(X),0)$.
Let $M$ be the bigraded vector space given by $M^{0,q}=\Ker(d_1^{0,q})$ for all $q\geq 0$ and $M^{p,q}=0$ for all $p>0$ and all $q\geq 0$,
where $d_1^{p,q}:E_1^{p,q}(X)\to E_1^{p+1,q}(X)$ denotes the differential of $E_1^{*,*}(X)$.
Then $M^{*,*}$ is a bigraded sub-complex of $(E_1^{*,*}(X),d_1)$ with trivial differential. Denote by $\varphi:(M,0)\to (E^{*,*}_1(X),d_1)$ the inclusion.
Since $\Ker(d_1^{0,*})\times \Ker(d_1^{0,*})\subset \Ker(d_1^{0,*})$,
the multiplicative structure induced by $\varphi$ on $M$ is closed in $M$.
Hence $\varphi$ is an inclusion of differential bigraded algebras.
On the other hand, we have $E_2^{0,q}(X)\cong M^{0,q}=\Ker(d_1^{0,q})$. This gives an inclusion of bigraded algebras
$\psi:M\to E^{*,*}_2(X)$.
Assume that the weight filtration on $H^k(X;\QQ)$ is pure of weight $k$, for all $k\leq r$.
We next show that both $\varphi$ and $\psi$ are $r$-quasi-isomorphisms. Indeed, 
for every $p>0$ and every $q\geq 0$ such that $p+q\leq r$ we have
$E_2^{p,q}(X)=0$, while for every $q\geq 0$,
$E_2^{0,q}(X)\cong M^{0,q} \cong H^q(X;\QQ)$.
Therefore the induced maps $H^k(\varphi)$ and $H^k(\psi)$ are isomorphisms for all $k\leq r$
and the maps $H^{k+1}(\varphi)$ and $H^{k+1}(\psi)$ are monomorphisms.
\end{proof}

We highlight the two extreme cases $r=1$ and $r=\infty$ in the following corollary. 
\begin{cor}\label{corolari}
 Let $X$ be a complex projective variety.
 \begin{enumerate}[(1)]
  \item If the weight filtration on $H^1(X;\QQ)$ is pure of weight $1$, then $\pi_1(X)$ is formal.
  \item  If the weight filtration on $H^k(X;\QQ)$ is pure of weight $k$, for all $k\geq 0$ then $X$ is formal.
 \end{enumerate}
  \end{cor}

\begin{example}[$\QQ$-homology manifolds]
Let $X$ be a complex projective variety of dimension $n$.
Assume that $X$ is a $\QQ$-homology manifold (for all $x\in X$, $H^k_{\{x\}}(X;\QQ)=0$ for $k\neq 2n$ and
$H^{2n}_{\{x\}}(X;\QQ)\cong \QQ$). Then the weight filtration on $H^k(X;\QQ)$ is pure of weight $k$, for all $k\geq 0$
(see Theorem 8.2.4 in \cite{DeHIII}). Hence $X$ is formal.
Examples of such varieties are given by weighted projective spaces or more generally $V$-manifolds (see Appendix B of \cite{Di}),
surfaces with $A_1$-singularities, the Cayley cubic or the Kummer surface.
\end{example}

In fact, purity of the weight filtration is strongly related to Poincar\'{e} duality: 
if $X$ is a complex projective variety whose rational cohomology satisfies Poincar\'{e} duality, then
the weight filtration on $H^k(X;\QQ)$ is pure of weight $k$, for all $k\geq 0$.
Indeed, the Poincar\'{e} duality maps
$H^k(X;\QQ)\cong (H^{2n-k}(X;\QQ)^\du)(-n)$ are compatible with mixed Hodge structures.
The weights on the left (resp. right) hand side are $\leq k$ (resp. $\geq k$),
hence equal to $k$.
Therefore such varieties are formal (cf. \cite{HaBull}, Theorem 5). 
Another well-known result relating purity and Poincar\'{e} duality is the purity of the weight filtration
on the (middle perversity) intersection cohomology $IH^*(X;\QQ)$ of a projective variety $X$.
Furthermore, Weber \cite{Weber} showed that for a complex projective variety
$X$, the image of the map $H^k(X;\QQ)\to IH^k(X;\QQ)$ is isomorphic to the pure term $Gr_k^WH^k(X;\QQ)$.

The purity of the weight filtration in cohomology does not imply Poincar\'{e} duality, as
shown by the following example.

\begin{example}[Projective cone over a smooth curve]
Let $C\subset \CC\PP^N$ be a smooth curve of genus $g$ and consider the projective cone $X=P_cC$ over $C$.
The Betti numbers of $X$ are $b_0=1$, $b_1=0$, $b_2=1$, $b_3=2g$ and $b_4=1$.
The weight filtration on $H^k(X;\QQ)$ is pure of weight $k$, for all $k\geq 0$, and hence $X$ is formal, but 
$H^*(X;\QQ)$ does not satisfy Poincar\'{e} duality.
\end{example}

In \cite{ADH}, it is shown that the weight filtration on $H^1(X;\QQ)$ of a normal complex projective variety $X$ is pure of weight 1.
In the case of isolated singularities, the proof is a standard argument in mixed Hodge theory:

\begin{lem}\label{normalpureH1}
Let $X$ be a normal complex projective variety with isolated singularities. Then the weight filtration on $H^1(X;\QQ)$ is pure weight 1.
\end{lem}
\begin{proof}
Let $\Sigma$ denote the singular locus of $X$ and $f:\wt X\to X$ a resolution such that $D:=f^{-1}(\Sigma)$ is a
simple normal crossings divisor.
Since $X$ is normal, by Zariski's main Theorem
we have $H^0(D;\QQ)\cong H^0(\Sigma;\QQ)$. Since $\dim\Sigma=0$, $H^k(\Sigma)=0$ for all $k>0$. This gives a Mayer-Vietoris long exact sequence
$$0\to H^1(X;\QQ)\xra{f^*}H^1(\wt X;\QQ)\xra{j^*}H^1(D;\QQ)\lra H^2(X;\QQ)\xra{f^*}\cdots $$
which is strictly compatible with the weight filtration (see for example 
Corollary-Definition 5.37 of \cite{PS}).
Since $\wt X$ is smooth and projective, its weight filtration is pure.
Hence the weight filtration on $H^1(X;\QQ)$ is pure weight 1.
\end{proof}

\begin{cor}\label{H1formal}The fundamental group $\pi_1(X)$ of every
normal complex projective variety $X$ with isolated singularities is formal.
\end{cor}

\section{Formality of Projective varieties with isolated singularities}\label{Sisolated}
By purely topological reasons we know that every simply connected, 4-dimensional CW-complex
is formal. We also know there exist non-formal 4-dimensional CW-complexes. 
As mentioned in the introduction, thanks to deep results of
Simpson and Kapovich-Koll\'{a}r we know that there exist non-formal complex projective surfaces.
In this section we prove that every complex projective surface with normal singularities is formal.
We generalize this result in two directions.
First, we prove formality for projective varieties of dimension $n$ with only
isolated singularities whose link is $(n-2)$-connected. Second, we prove formality for those
projective varieties with normal isolated singularities admitting a resolution of singularities
with smooth exceptional divisor.

\begin{teo}\label{formalsurfaces}
Every normal complex projective surface is a formal topological space.
\end{teo}
\begin{proof}
Let $X$ be a normal complex projective surface, which we may assume to be connected.
We use the formulas for $E_1^{*,*}(X)$ given in Proposition $\ref{isolsingE1}$.
Since normal singularities have codimension $\geq 2$ we have $\dim(\Sigma)=0$.
Since $\dim(D^{(p)})=1-p$ we have $E_1^{p,q}(X)=0$ for all $q>4-2p$.
By Lemma $\ref{normalpureH1}$ together with semi-purity, the weight filtration
on $H^k(X;\QQ)$ is pure of weight $k$ for all $k\neq 2$.
We have:
$$
E_1^{*,*}(X)=
\def\arraystretch{1.5}
\begin{array}{|l}
E^{0,4}_1\\
E^{0,3}_1\\
E^{0,2}_1\xra{d^{0,2}_1}E^{1,2}_1\\
E^{0,1}_1\xra{d^{0,1}_1}E^{1,1}_1\\
E^{0,0}_1\xra{d^{0,0}_1}E^{1,0}_1\xra{d^{1,0}_1}E^{2,0}_1\\
\hline
\end{array}
\Longrightarrow
E_2^{*,*}(X)\cong 
\def\arraystretch{1.5}
\begin{array}{|ccc }
H^4(\wt X;\QQ)\\
H^3(\wt X;\QQ)\\
\Ker(d_1^{0,2})& 0\\
\Ker(d_1^{0,1})&\Coker(d_1^{0,1})\\
\Ker(d_1^{0,0}) &0 &\Coker(d_1^{1,0})\\
\hline
\end{array}
$$

For all $p,q\geq 0$ take a section $E_2^{p,q}(X)\to \Ker(d_1^{p,q})\subset E^{p,q}_1(X)$ of the projection
$\Ker(d^{p,q}_1)\twoheadrightarrow E_2^{p,q}(X)$.
This defines a morphism $\rho:(E^{*,*}_2(X),0)\to (E^{*,*}_1(X),d_1)$ 
of bigraded complexes which is a quasi-isomorphism. We next show that $\rho$ is multiplicative.
By bidegree reasons, the only non-trivial products in $E_2^{*,*}(X)$ are the products
$$\Ker(d_1^{0,q})\times \Ker(d_1^{0,q'})\lra \Ker(d_1^{0,q+q'})$$
induced by the cup product of $H^*(\wt X;\QQ)$.
Since $\rho$ is the identity on $\Ker(d_1^{0,q})$, it preserves these products.
It also preserves the unit $1\in\Ker(d_1^{0,0})$.
It only remains to see that
the diagram
$$
\xymatrix{
\Ker(d_1^{0,1})\times \Coker(d_1^{0,1})\ar[d]^{\rho\times\rho}\ar[r]&0\ar[d]\\
E_1^{0,1}(X)\times E_1^{1,1}(X)\ar[r]&E_1^{1,2}(X)
}
$$
commutes. 
By Proposition $\ref{isolsingE1}$, the term $E_1^{0,1}(X)$ is given by the pull-back
$$
\xymatrix{\pb
\ar[d]
E_1^{0,1}(X)\ar[r]&H^1(D^{(1)};\QQ)\otimes\Lambda(t)t\ar[d]^{\delta^1}\\
H^1(\wt X;\QQ)\ar[r]^-{j^*_1}&H^1(D^{(1)};\QQ)
}
$$
and 
$E_1^{1,k}(X)\cong H^{k}(D^{(1)};\QQ)\otimes\Lambda(t)dt$ for $k\in\{1,2\}$.
The differential $d_1^{0,1}:E_1^{0,1}(X)\to E_1^{1,1}(X)$ is given by $(x,a(t))\mapsto a'(t)dt$ and the product 
$E_1^{0,1}(X)\times E_1^{1,1}(X)\to E_1^{1,2}(X)$
is given by $(x,a(t))\cdot b(t)dt=a(t)b(t)dt$.

Let $(x,a(t))\in E_1^{0,1}(X)$. Since $a(0)=0$ it follows that 
$(x,a(t))\in\Ker(d_1^{0,1})$ if and only if $a(t)=0$.
Therefore we have $\Ker(d_1^{0,1})\cdot E_1^{1,1}(X)=0$, and the above diagram commutes.
This proves that the map $\rho:(E^{*,*}_2(X),0)\to (E^{*,*}_1(X),d_1)$ is multiplicative. 
Since $X$ is connected, it has a mixed Hodge diagram $\Aa(X)$ with $H^0(\Aa(X)_\QQ)\cong \QQ$. Hence
by Theorem 
$\ref{E1formal}$ we have a string of quasi-isomorphisms of complex cdga's
$$(\Aa_{pl}(X),d)\otimes\CC\stackrel{\sim}{\longleftrightarrow}(E_1^{*,*}(X),d_1)\otimes\CC \stackrel{\sim}{\longleftarrow}(E_2^{*,*}(X),0)\otimes\CC\cong (H^{*}(X;\CC),0).$$
To conclude that $X$ is formal it suffices to apply descent of formality of cdga's from $\CC$ to $\QQ$.
\end{proof}

The following is an example of a normal projective surface with isolated singularities and non-trivial weight filtration on $H^2$
(cf. $\S7$ of \cite{ToChow}, see also \cite{BVS}).

\begin{example}
Let $C$ be a curve of degree $d\geq 3$ with $n> 0$ nodes in $\CC\PP^2$. The genus of $C$ is given by
$g=(d-1)(d-2)/2-n$. 
Choose a smooth projective curve $C'$ of
degree $d'=d+1$ intersecting $C$ transversally at smooth points of $C$, so that $|C\cap C'|=dd'$
and consider the blow-up $\wt X=Bl_{C\cap C'}\CC\PP^2$ of $\CC\PP^2$ at the $dd'$ points of $C\cap C'$.
Then the proper transform $\wt C$ of $C$ has negative self-intersection $|\wt C\cap \wt C|=d(d-d')=-d$
and we may consider the blow-down $X$ of $\wt C$ to a point. 
Explicitly, assume that the curve $C$ is given by $f(x,y,z)=0$, and that $C'$
is given by $g(x,y,z)=0$. Then $X$ is the projective variety
defined by the equation $wf(x,y,z)+g(x,y,z)=0,$
which has a normal isolated singularity at $(0,0,0,1)$.
Here $(x,y,z,w)$ are the homogeneous coordinates in $\CC\PP^3$.
The normalization of $\wt C$ is a smooth projective curve of genus $g$ and $\wt X$ is homeomorphic to the connected sum of
$dd'+1$ projective planes.
Deligne's weight spectral sequence can be written as:

$$
\wt E_1^{*,*}(X)
\cong
\def\arraystretch{1.2}
\begin{array}{|lllll}
\QQ\\
0\\
\QQ^{dd'+1}&\twoheadrightarrow&\QQ\\
0&&\QQ^{2g}\\
\QQ\oplus \QQ&\twoheadrightarrow&\QQ&\stackrel{0}{\to}&\QQ^n\\
\hline
\end{array}
\Longrightarrow 
E_2^{*,*}(X)=
\def\arraystretch{1.2}
\begin{array}{|lllll}
\QQ\\
0\\
\QQ^{dd'}&0 \\
0&\QQ^{2g}\\
\QQ&0&\QQ^n\\
\hline
\end{array}.
$$
Hence $H^2(X;\QQ)$ has a non-trivial weight filtration:
$$Gr^W_0H^2(X;\QQ)\cong \QQ^{n}, Gr^W_1H^2(X;\QQ)\cong \QQ^{2g} \text{ and }Gr^W_2H^2(X;\QQ)\cong \QQ^{dd'}.$$
Since $X$ is simply connected (see for example Corollary V.2.4 of \cite{Di}), we may compute the rational homotopy groups of $X$ with their weight filtration from a 
bigraded minimal model $\rho:M\stackrel{\sim}{\lra} E^{*,*}_2(X)$ of the bigraded algebra $E^{*,*}_2(X)$. 
The weight filtration on $\pi_i:=\pi_i(X)\otimes \QQ$
satisfies $Gr^W_q \pi_{p+q}\cong \Hom(Q(M)^{p,q},\QQ)$,
where $Q(M)^{p,q}$ denotes the indecomposables of $M$ of bidegree $(p,q)$.

The cohomology ring of $\wt X$ is given by
$H^{*}(\wt X;\QQ)\cong \QQ[a,b_1,\cdots,b_{dd'}]$ with 
$a^2=T$, $b_i^2=-T$ and $b_i\cdot b_j=0$ for all $i\neq j$. Here 
$T$ denotes the top class of $\wt X$, $a$ is the hyperplane class and $b_i$ correspond to the exceptional divisors.
Let $\gamma_i:=a-d\cdot b_i$. Then 
$E_2^{0,2}(X)\cong \Ker(d^{0,2}_1)\cong\QQ[\gamma_1,\cdots,\gamma_{dd'}]$ with
$\gamma_i^2=T(1-d^2)$ and
$\gamma_i\cdot \gamma_j=d^2T$ for $i\neq j$.
Hence we may write
$$E^{*,*}_2(X) \cong \QQ[\alpha_1,\cdots,\alpha_n,\beta_1,\cdots,\beta_{2g},\gamma_1,\cdots,\gamma_{dd'}]$$
where the generators have bidegree
$|\alpha_i|=(2,0)$, $|\beta_i|=(1,1)$ and $|\gamma_i|=(0,2)$.
By bidegree reasons, the only non-trivial products are given by
$\gamma_i^2=T(1-d^2)$ and 
$\gamma_i\cdot \gamma_j=d^2T$, for all $i\neq j$.
We compute the first steps of a minimal model for $E^{*,*}_2(X)$. Let $M_2$ be the free bigraded algebra
$$M_2=\Lambda\left(\ov\alpha_1,\cdots,\ov\alpha_n,\ov\beta_1,\cdots,\ov\beta_{2g},\ov\gamma_1,\cdots,\ov\gamma_{dd'}\right).$$
with trivial differential generated by elements of bidegree
$|\ov \alpha_i|=(2,0)$, $|\ov\beta_i|=(1,1)$ and $|\ov\gamma_i|=(0,2)$.
Then the map $\rho_2:M_2\lra E^{*,*}_2(X)$ given by $\ov x\mapsto x$ is a
$2$-quasi-isomorphism of bigraded algebras. Hence we have
$$Gr^W_0\pi_2\cong \QQ^n,\, Gr^W_1\pi_2\cong \QQ^{2g}\text{ and }Gr^W_2\pi_2\cong \QQ^{dd'}.$$
Let $M_3=M_2\otimes_d\Lambda(V_{3,0},V_{2,1},V_{1,2},V_{0,3},V_{-1,4})$ where
$V_{i,j}$ are the graded vector spaces of pure bidegree $(i,j)$, and $d:V_{i,j}\lra M_2^{i+1,j}$ are the differentials
given by:
$$V_{3,0}=\QQ\langle x_{ij}\rangle,\,dx_{ij}=\ov \alpha_i\ov\alpha_j;\,1\leq i\leq j\leq n.$$
$$V_{2,1}=\QQ\langle y_{ij}\rangle,\,dy_{ij}=\ov \alpha_i\ov\beta_j;\, 1\leq i\leq n, 1\leq j\leq 2g.$$
$$V_{1,2}=\QQ\langle z_{ij},w_{kl}\rangle,\, dz_{ij}=\ov \alpha_{i}\ov \gamma_j, dw_{kl}=\ov \beta_{k}\ov \beta_l; 1\leq i\leq n, 1\leq j\leq dd', 1\leq k\leq l\leq 2g.$$
$$V_{0,3}=\QQ\langle \tau_{ij}\rangle,\,d\tau_{ij}=\ov \beta_i\ov\gamma_j;\, 1\leq i\leq n, 1\leq j\leq 2g.$$
$$V_{-1,4}=\QQ\langle \xi_{ij}\rangle,\,d\xi_{ij}=\ov \gamma_i\ov\gamma_j;\, 1\leq i\leq dd', (i,j)\neq(1,1).$$
Then the extension $\rho_3:M_3\to E^{*,*}_2(X)$ of $\rho_2$ given by $V_{i,j}\mapsto 0$ is a $3$-quasi-isomorphism.
The formula
$Gr^W_p\pi_3\cong \Hom(V_{3-p,p},\QQ)$ gives:
$$
Gr^W_0\pi_3\cong \QQ^{{n(n+1)}\over 2}, Gr^W_1\pi_3\cong\QQ^{2g\cdot n}, Gr^W_2\pi_3\cong \QQ^{dd'\cdot n+g(2g+1)},Gr^W_3\pi_3\cong \QQ^{dd'\cdot 2g},Gr^W_4\pi_3\cong \QQ^{{dd'(dd'+1)\over 2}-1}.
$$

For example, we may take $C$ to be the nodal cubic curve given by $f(x,y,z)=y^2z-x^2z-x^3$
and $C'$ a smooth plane quartic. Then $dd'=12$, $g=0$ and $n=1$. This gives:
$$
\def\arraystretch{1.5}
\begin{array}{l}
Gr^W_0\pi_2\cong \QQ,\, Gr^W_1\pi_2=0\text{ and }Gr^W_2\pi_2\cong \QQ^{12}.\\
Gr^W_0\pi_3\cong \QQ, Gr^W_1\pi_3=0, Gr^W_2\pi_3\cong \QQ^{12},Gr^W_3\pi_3=0 \text{ and }
Gr^W_4\pi_3\cong \QQ^{77}.
\end{array}
$$
\end{example}

The following is a generalization of Theorem $\ref{formalsurfaces}$ to projective varieties of arbitrary dimension.

\begin{teo}\label{linkconnectedformal}
 Let $X$ be a complex projective variety of dimension $n$ with normal isolated singularities. 
 Denote by $\Sigma$ the singular locus of $X$, and for each $\sigma\in \Sigma$ let
 $L_\sigma$ denote the link of $\sigma$ in $X$. If $\wt H^k(L_\sigma;\QQ)=0$ for all $k\leq n-2$ 
 for every $\sigma\in \Sigma$, then $X$ is a formal topological space.
\end{teo}
\begin{proof}
The link $L_\sigma$ of $\sigma\in\Sigma$ in $X$ is a smooth connected real manifold of dimension $2n-1$.
Let $L=\sqcup_{\sigma\in \Sigma} L_\sigma$. Then $H^0(L;\QQ)\cong H^0(\Sigma;\QQ)$.
Assume that $\wt H^k(L_\sigma;\QQ)=0$ for all $k\leq n-2$. By Poincar\'{e} duality the only non-trivial rational cohomology groups of $L$ are
in degrees $0,n-1,n$ and $2n-1$. Let $X_{reg}=X-\Sigma$. From the Mayer-Vietoris exact sequence
$$\cdots\to H^{k-1}(L;\QQ)\to H^k(X;\QQ)\to H^{k}(X_{reg};\QQ)\oplus H^k(\Sigma;\QQ)\to H^k(L;\QQ)\to\cdots$$
it follows that the map $H^k(X;\QQ)\lra H^k(X_{reg};\QQ)$ is an isomorphism whenever $k<n-1$ or $n+1<k<2n-1$,
and injective for $k=n-1$.
Since $H^k(X;\QQ)$ has weights in $\{0,1,\cdots,k\}$ and $H^k(X_{reg};\QQ)$ has weights in $\{k,k+1,\cdots,2k\}$,
and the morphism $H^k(X;\QQ)\to H^k(X_{reg};\QQ)$ is strictly compatible with the weight filtrations,
it follows that for $k\neq n,n+1$, the weight filtration on $H^k(X;\QQ)$ is pure of weight $k$.
Furthermore, by semi-purity we have that $H^{n+1}(X;\QQ)$ is pure of weight $n+1$.
Therefore the only non-trivial weights of $H^*(X;\QQ)$ are in degree $k=n$.
The weight spectral sequence for $X$ has the form
$$
E_1^{*,*}(X)=
\def\arraystretch{1.2}
\begin{array}{cccccc}
\text{\tiny{$2n$}}\\
\\
\\
\\
\text{\tiny{$n$}}\\
\\
\\
\\
\text{\tiny{$0$}}\\
\multicolumn{1}{c}{}
\end{array}
\def\arraystretch{1.2}
\begin{array}{|cccccc}
\bullet&\gris&\gris&\gris&\gris\\
\bullet&\gris&\gris&\gris&\gris\\
\bullet&\bullet&\gris&\gris&\gris\\
\bullet&\bullet&\gris&\gris&\gris\\
\bullet&\bullet&\bullet&\gris&\gris\\
\bullet&\bullet&\bullet&\gris&\gris\\
\bullet&\bullet&\bullet&\bullet&\gris\\
\bullet&\bullet&\bullet&\bullet&\gris\\
\bullet&\bullet&\bullet&\bullet&\bullet\\
\hline
\multicolumn{1}{c}{}&\multicolumn{1}{c}{}&\multicolumn{1}{c}{}&\multicolumn{1}{c}{}&\multicolumn{1}{c}{\text{\tiny{$n$}}}
\end{array}
\Longrightarrow
E_2^{*,*}(X)=
\def\arraystretch{1.2}
\begin{array}{cccccc}
\text{\tiny{$2n$}}\\
\\
\\
\\
\text{\tiny{$n$}}\\
\\
\\
\\
\text{\tiny{$0$}}\\
\multicolumn{1}{c}{}
\end{array}
\def\arraystretch{1.2}
\begin{array}{|cccccc}
\bullet&\gris&\gris&\gris&\gris\\
\bullet&\gris&\gris&\gris&\gris\\
\bullet&&\gris&\gris&\gris\\
\bullet&&\gris&\gris&\gris\\
\bullet&&&\gris&\gris\\
\bullet&\bullet&&\gris&\gris\\
\bullet&&\bullet&&\gris\\
\bullet&&&\bullet&\gris\\
\bullet&&&&\bullet\\
\hline
\multicolumn{1}{c}{}&\multicolumn{1}{c}{}&\multicolumn{1}{c}{}&\multicolumn{1}{c}{}&\multicolumn{1}{c}{\text{\tiny{$n$}}}
\end{array}
$$
where the bullets denote the non-trivial elements.
Consider the quasi-isomorphism of complexes $\rho:(E^{*,*}_2(X),0)\to (E^{*,*}_1(X),d_1)$ defined by taking sections of the projections $\Ker(d^{p,q}_1)\twoheadrightarrow E_2^{p,q}(X)$.
We next show that $\rho$ is multiplicative.
Note that by bidegree reasons, the only non-trivial products of $E_2^{*,*}(X)$ are between elements of the first column. Since $\rho$ is the identity on 
$E_2^{0,*}(X)\cong \Ker(d_1^{0,q})$,
it preserves these products. It also preserves the unit $1\in E_2^{0,0}(X)$.
Since $E_1^{p,q}(X)=0$ for all $q>2(n-p)$, we have $E^{p,n-p}_1(X)\cdot E^{p',n-p'}_1(X)=0$ for all $p,p'>0$.
Therefore it only remains to show that for $p,q>0$, the following diagram commutes.
$$
\xymatrix{
E_2^{0,q}(X)\times E_2^{p,n-p}(X)\ar[d]^{\rho\times\rho}\ar[r]&0\ar[d]\\
E_1^{0,q}(X)\times E_1^{p,n-p}(X)\ar[r]&E_1^{p,n-p+q}(X)
}
$$
By Proposition $\ref{isolsingE1}$ and since $H^q(\Sigma)=0$ for $q>0$, the term $E_1^{0,q}(X)$ is given by
$$
\xymatrix{
\pb\ar[d]
E_1^{0,q}(X)\ar[r]&E_1^{0,q}(D)\otimes\Lambda(t)\cdot t\ar[d]^{\delta^1}\\
H^1(\wt X;\QQ)\ar[r]^{j_1^*}&E_1^{0,q}(D)
},
$$
while for $p>0$ we have $E^{p,n-p}_1(X)\cong E_1^{p,n-p}(D)\otimes\Lambda(t)\oplus E_1^{p-1,n-p}(D)\otimes\Lambda(t)dt$.
The proof now follows as in the proof of Theorem $\ref{formalsurfaces}$.
\end{proof}

\begin{example}[Complete intersections]
Let $X$ be a complete intersection of dimension $n>1$.
Assume that the singular locus $\Sigma=\{\sigma_1,\cdots,\sigma_N\}$ is a finite number of points. 
The link 
of $\sigma_i$ in $X$ is $(n-2)$-connected
(this result is due to Milnor \cite{Milnor} in the case of hypersurfaces
and to Hamm \cite{Hamm} for general complete intersections).
Therefore by Theorem $\ref{linkconnectedformal}$, $X$ is formal.
Note that in particular, every complex hypersurface with isolated singularities is formal.
\end{example}

\begin{teo}\label{smoothdivisor}
Let $X$ be a projective variety with only isolated singularities. Assume that there exists a resolution of singularities $f:\wt X\to X$ such that
the exceptional divisor $D=f^{-1}(\Sigma)$ is smooth. Then $X$ is a formal topological space.
\end{teo}
\begin{proof}
We may assume that $X$ is connected.
By Proposition $\ref{isolsingE1}$ the multiplicative weight spectral sequence is given by
$$
E_1^{*,*}(X)
\cong
\def\arraystretch{1.5}
\begin{array}{|lllll}
E_1^{0,q}(X)&\lra&H^q(D)\otimes\Lambda(t)dt\\
E_1^{0,0}(X)&\lra& H^0(D)\otimes\Lambda(t)dt\\
\hline
\end{array}
\def\arraystretch{1.5}
\begin{array}{c}
\text{\tiny{$q\geq 1$}}\\
\text{\tiny{$q=0$}}\\
\end{array}
$$
where
$$
\xymatrix{
\pb\ar[d]
E_1^{0,0}(X)\ar[r]&H^0(D)\otimes\Lambda(t) \ar[d]^{\delta^1}\\
H^0(\wt X)\ar[r]^{j^*} &H^0(D)
}\text{ and }
\xymatrix{
\pb\ar[d]
E_1^{0,q}(X)\ar[r]&H^q(D)\otimes\Lambda(t)t \ar[d]^{\delta^1}\\
H^q(\wt X)\ar[r]^{j^*} &H^q(D)
}\text{ for }q>0.
$$
The differential $d_1:E_1^{0,*}(X)\lra E_1^{1,*}(X)$ is given by $(x,a(t))\mapsto a'(t)dt$.
The non-trivial products of $E_1^{*,*}(X)$ are the maps
$E_1^{0,q}(X)\times E_1^{0,q'}(X)\lra E_1^{0,q+q'}(X)$
given by $(x,a(t))\cdot (y,b(t))=(x\cdot y,a(t)\cdot b(t))$
and the maps
$E_1^{0,q}(X)\times E_1^{1,q'}(X)\lra E_1^{1,q+q'}(X)$
given by $(x,a(t))\cdot b(t)dt=a(t)\cdot b(t)dt$.
The unit is $(1_{\wt X},1_D)\in E_1^{0,0}(X)$.
By computing the cohomology of $E_1^{*,*}(X)$ we find:
$$
E_2^{*,*}(X)
\cong
\def\arraystretch{1.5}
\begin{array}{|lllll}
H^{2n}(\wt X)&0\\
H^{2n-1}(\wt X)&0\\
\Ker(j^*)^{2n-2}&\Coker(j^*)^{2n-2}\\
\vdots &\vdots\\
\Ker(j^*)^1&\Coker(j^*)^1\\
H^0(\wt X)& 0\\
\hline
\end{array}
$$
with the non-trivial products being $\Ker(j^*)^q\times \Ker(j^*)^{q'}\lra \Ker(j^*)^{q+q'}$.
Define a quasi-isomorphism $\rho:(E_2^{*,*}(X),0)\lra (E_1^{*,*}(X),d_1)$ of bigraded complexes as follows.
Let $\rho: E_2^{0,q}(X)\lra E_1^{0,q}(X)$ be defined by the inclusion, for all $q\geq 0$.
Define $\rho: E_2^{1,q}(X)\lra E_1^{1,q}(X)$ by taking a section
$\Coker(j^*)^{q}\to H^q(D)$ of the projection
$H^q(D)\twoheadrightarrow \Coker(j^*)^{q}$, for all $q>0$.

To see that $\rho$ is a morphism of bigraded algebras it suffices to observe that it preserves
the unit and that
$\rho(\Ker(j^*))\cdot E_1^{1,*}(X)=0$.
Since $X$ is connected, it has a mixed Hodge diagram $\Aa(X)$ with $H^0(\Aa(X)_\QQ)\cong \QQ$.
Hence by Theorem
$\ref{E1formal}$ we have a string of quasi-isomorphisms of complex cdga's
$$(\Aa_{pl}(X),d)\otimes\CC\stackrel{\sim}{\longleftrightarrow}(E_1^{*,*}(X),d_1)\otimes\CC \stackrel{\sim}{\longleftarrow}(E_2^{*,*}(X),0)\otimes\CC\cong (H^{*}(X;\CC),0).$$
To conclude that $X$ is formal it suffices to apply descent of formality of cdga's from $\CC$ to $\QQ$.
\end{proof}

\begin{example}
Projective varieties with only isolated ordinary multiple points are formal. Projective cones over smooth projective varieties are formal.
\end{example}

\begin{example}[Segre cubic]
The Segre cubic $S$ is a simply connected projective threefold with 10 singular points, and is described by 
the set of points $(x_0:x_1:x_2:x_3:x_4:x_5)$ of $\CC\PP^5$
$$S: \left\{x_0+x_1+x_2+x_3+x_4+x_5=0,\, x_0^3+x_1^3+x_2^3+x_3^3+x_4^3+x_5^3=0\right\}.$$
A resolution $f:\ov{\Mm}_{0,6}\lra S$ of $S$ is given by the moduli space $\ov{\Mm}_{0,6}$ of stable rational curves with 6 marked points,
and $f^{-1}(\Sigma)=\bigsqcup_{i=1}^{10} \CC\PP^1\times\CC\PP^1$, where $\Sigma=\{\sigma_1,\cdots,\sigma_{10}\}$ is the singular locus of $S$.
We have
$$
\wt E_1^{*,*}(S)\cong
\def\arraystretch{1.2}
\begin{array}{|lll}
\QQ&\\
0&\\
\QQ^{16}&\twoheadrightarrow&\QQ^{10}\\
0&&0\\
\QQ^{16}&\to&\QQ^{20}\\
0&&0\\
\QQ^{11}&\twoheadrightarrow&\QQ^{10}\\
\hline
\end{array}
\Longrightarrow 
E_2^{*,*}(S)=
\def\arraystretch{1.2}
\begin{array}{|ll}
\QQ&\\
0&\\
\QQ^6&0\\
0&0\\
\QQ&\QQ^5\\
0&0\\
\QQ&0\\
\hline
\end{array}.
$$
Hence $S$ has a non-trivial weight filtration, with $0\neq Gr^W_2H^3(S;\QQ)\cong \QQ^5$.
By Theorem $\ref{smoothdivisor}$ and since $S$ is simply connected,
we may compute the rational
homotopy groups $\pi_*(S)\otimes \QQ$ with their weight filtration
from a minimal model of $E_2^{*,*}(S)$.
Since $S$ is a hypersurface of $\CC\PP^4$ the map $S\to \CC\PP^4$ induces an isomorphism $H^k(\CC\PP^4)\cong H^k(S)$ for $k\neq 3,4$
(see Theorems V.2.6 and V.2.11 of \cite{Di}).
We deduce that
$E_2^{*,*}(S)\cong \QQ[a,b_1,\cdots,b_5,c_0,\cdots,c_5,e]$
with the only non-trivial products $a^2=c_0$ and $a^3=e$. The bidegrees are given by
$|a|=(0,2)$, $|b_i|=(1,2)$, $|c_i|=(0,4)$ and $e=(0,6)$.
In low degrees we obtain:
$$
\def\arraystretch{1.5}
\begin{array}{l}
Gr_2^W\pi_2\cong\QQ, Gr_2^W\pi_3\cong\QQ^5, Gr_4^W\pi_4\cong\QQ^5, Gr_3^W\pi_5\cong\QQ^{10}, Gr_5^W\pi_5\cong\QQ^5,\\
Gr_4^W\pi_6\cong\QQ^{25}, Gr_5^W\pi_6\cong\QQ^{25}, Gr_4^W\pi_7\cong\QQ^{40}, Gr_5^W\pi_7\cong\QQ^{50},Gr_7^W\pi_7\cong\QQ^{26}.  
  \end{array}
$$
\end{example}

\section{Contractions of subvarieties}\label{contractions}
Let $Y\hookrightarrow X$ be a closed immersion of complex projective varieties. Assume that $Y$ contains the singular locus of $X$ and 
denote by $X/Y$ the space obtained by contracting each connected component of $Y$ to a point.
In general, $X/Y$ is not an
algebraic variety. For instance, the contraction of a rational curve in a smooth projective surface is a complex algebraic variety if and only if 
the self-intersection number of the curve is negative.
For contractions of divisors in higher-dimensional varieties there are general conditions on the conormal line bundle, which ensure the existence of contractions 
in the categories of complex analytic spaces and complex algebraic varieties respectively (see for example \cite{Grauert}, \cite{Ar2}).
In the general situation in which $X/Y$ is a pseudo-manifold with normal isolated singularities, we can endow the rational homotopy
type of $X/Y$ with a mixed Hodge structure, coming from the mixed Hodge structures on 
$X$ and $Y$ as follows.

\begin{prop}\label{MHScontractions}
Let $Y\hookrightarrow X$ be a closed immersion of complex projective varieties. 
Assume that $Y$ contains the singular locus of $X$ and that $X$ is connected.
\begin{enumerate}[(1)]
 \item The rational homotopy type of $X/Y$ carries mixed Hodge structures. 
 \item Let $f:\wt X\lra X$ be a resolution of $X$ such that $D:=f^{-1}(Y)$ is a simple normal crossings divisor.
Denote by $\Sigma$ the singular locus of $X/Y$.
 There is a string of quasi-isomorphisms
from $\Aa_{pl}(X/Y)\otimes\CC$ to $E^{*,*}_1(X/Y)\otimes\CC$, where $E^{*,*}_1(X/Y)$ is the bigraded algebra given by
$$E_1^{*,*}(X/Y)=s_{\TW}\left(H^*(\Sigma;\QQ)\times H^*(\wt X;\QQ)\rightrightarrows E_1^{*,*}(D)\right).$$
\end{enumerate}
\end{prop}
\begin{proof}
Take mixed Hodge diagrams $\Aa(\Sigma)$ and $\Aa(\wt X)$ with trivial weight filtrations, and let $\Aa(D)$ be a mixed Hodge diagram for $D$, as
defined in the proof of Proposition $\ref{ssdivisor}$.
Let
$$\Aa(X/Y):=s_{\TW}\left(\Aa(\Sigma)\times \Aa(\wt X)\rightrightarrows \Aa(D)\right).$$
By cohomological descent
we have $\Aa(X/Y)_\QQ\simeq \Aa_{pl}(X/Y)$.
Indeed, since the category of cdga's with the Thom-Whitney simple and the class of quasi-isomorphisms 
is a cohomological descent category (see \cite{GN}, Proposition 1.7.4),
it suffices to show that $H^*(\wt X,D)\cong H^*(X/Y,\Sigma)$.
 This follows from excision, since the composition $\wt X\to X\to X/Y$ is an isomorphism outside $\Sigma$.
Hence 
by Lemma $\ref{TWisMHD}$, $\Aa(X/Y)$ is a mixed Hodge diagram
for $X/Y$. This proves (1).
(2) follows from Lemma $\ref{sssimple}$ and Theorem $\ref{E1formal}$.
\end{proof}

The weight filtration on the cohomology of a projective variety with isolated singularities has 
some properties special to the algebraic case, which are not satisfied in the general setting of contractions. For instance,
 the weight filtration on the cohomology $H^*(X/Y;\QQ)$ is not semi-pure in general.
Another feature of the weight filtration that is only applicable to algebraic varieties is the existence on a non-zero cohomology class $w\in Gr_2^WH^2(X;\QQ)$,
as a consequence of hard Lefschetz theory. Therefore the weight filtration on the cohomology of a contraction may serve as an obstruction theory 
for such contraction to be an algebraic variety.

\begin{example}
Consider two projective lines in $\CC\PP^2$ intersecting at a point, and 
let $X$ denote the topological space given by contraction the two lines to a point in $\CC\PP^2$.
$$
\xymatrix{
\ar[d]
\CC\PP^1\cup \CC\PP^1 \ar[d]\ar[r]&\CC\PP^2\ar[d]\\
\{\ast\}\ar[r]&X
}
$$
Then the weight spectral sequence is given by:
$$
\wt E_1^{*,*}(X)\cong
\def\arraystretch{1.2}
\begin{array}{|ccccc}
\QQ\\
0\\
\QQ&\lra&\QQ^2\\
0&&0\\
\QQ^2&\lra&\QQ^2&\lra&\QQ\\
\hline
\end{array}\Longrightarrow
\wt E_2^{*,*}(X)\cong
\def\arraystretch{1.2}
\begin{array}{|ccccc}
\QQ\\
0\\
0&\QQ\\
0&0\\
\QQ&0&0\\
\hline
\end{array}
$$
Since
$Gr_2^WH^3(X;\QQ)\neq 0$, the weight filtration on $H^*(X)$ is not semi-pure.
Furthermore, we have $H^2(X;\QQ)=0$.
This is a proof of the fact that $X$ is not algebraic.
Since $X$ is a simply connected topological space of dimension 4, it is formal (see for example \cite{FOT}, Proposition 2.99).
We may compute the rational homotopy groups $\pi_i:=\pi_i(X)\otimes \QQ$ with their weight filtration
from a minimal model of $E_2^{*,*}(X)\cong \Lambda(a_3,b_4)/(a_3\cdot b_4, b_4^2)$. 
In low degrees, we have:
$$\pi_2=0, Gr_2^W\pi_3\cong \QQ, Gr_4^W\pi_4\cong \QQ, \pi_5=0, 
Gr_6^W\pi_6\cong\QQ\text{ and }Gr_8^W\pi_7\cong\QQ.
$$

\end{example}

We now prove analogue statements of Theorems $\ref{purityformality}$,
$\ref{linkconnectedformal}$ and $\ref{smoothdivisor}$ for contractions of subvarieties.

\begin{teo}\label{formal_contractions}
Let $X$ be a connected complex projective variety of dimension $n$ and let
$Y\hookrightarrow X$ be a closed immersion such that $Y$ contains the singular locus of $X$.
If one of the following conditions is satisfied, then $X/Y$ is a formal topological space.
\begin{enumerate}[(a)]
\item The weight filtration on $H^k(X/Y)$ is pure of weight $k$, for each $k\geq 0$.
 \item The link $L_i:=L(\sigma_i,X/Y)$ of each singular point $\sigma_i\in X/Y$ in $X/Y$ is $(n-2)$-connected.
  \item The link $L_i':=L(Y_i,X)$ of each connected component $Y_i$ of $Y$ in $X$ is $(n-2)$-connected.
 \item There is a resolution of singularities $f:\wt X\to X$ such that $D=f^{-1}(X)$ is smooth.
\end{enumerate}
\end{teo}
\begin{proof}
Assume that $(a)$ is satisfied. By $(2)$ of Proposition $\ref{MHScontractions}$, 
the proof of Theorem $\ref{purityformality}$ is valid in this setting. Hence purity of the weight filtration implies formality.
Note that $(b)$ and $(c)$ are equivalent, since, the links $L_i=L(\sigma,X/Y)$ and $L_i':=L(Y_i,X)$
are homeomorphic (see Application 4.2 in \cite{Durfee}). Assume that $(b)$ is satisfied. Let $L=\bigsqcup_i L_i$.
From the Mayer-Vietoris exact sequence
$$\cdots\to  H^{k-1}(X/Y-\Sigma)\oplus H^{k-1}(\Sigma)\to H^{k-1}(L)\to H^k(X/Y)\to H^{k}(X/Y-\Sigma)\oplus H^k(\Sigma)\to H^k(L)\to\cdots$$
we find that
for $k\neq n, n+1$, the weight filtration on $H^k(X/Y;\QQ)$ is pure of weight $k$.
Note that since $X/Y$ is not algebraic in general, the weight filtration is not necessarily semi-pure.
Hence $H^{n+1}(X/Y;\QQ)$  may have non-trivial weights.
The weight spectral sequence for $X/Y$ has the form
$$
E_1^{*,*}(X/Y)=
\def\arraystretch{1.2}
\begin{array}{cccccc}
\text{\tiny{$2n$}}\\
\\
\\
\\
\text{\tiny{$n$}}\\
\\
\\
\\
\text{\tiny{$0$}}\\
\multicolumn{1}{c}{}
\end{array}
\def\arraystretch{1.2}
\begin{array}{|cccccc}
\bullet&\gris&\gris&\gris&\gris\\
\bullet&\gris&\gris&\gris&\gris\\
\bullet&\bullet&\gris&\gris&\gris\\
\bullet&\bullet&\gris&\gris&\gris\\
\bullet&\bullet&\bullet&\gris&\gris\\
\bullet&\bullet&\bullet&\gris&\gris\\
\bullet&\bullet&\bullet&\bullet&\gris\\
\bullet&\bullet&\bullet&\bullet&\gris\\
\bullet&\bullet&\bullet&\bullet&\bullet\\
\hline
\multicolumn{1}{c}{}&\multicolumn{1}{c}{}&\multicolumn{1}{c}{}&\multicolumn{1}{c}{}&\multicolumn{1}{c}{\text{\tiny{$n$}}}
\end{array}
\Longrightarrow
E_2^{*,*}(X/Y)=
\def\arraystretch{1.2}
\begin{array}{cccccc}
\text{\tiny{$2n$}}\\
\\
\\
\text{\tiny{$n+1$}}\\
\text{\tiny{$n$}}\\
\\
\\
\\
\text{\tiny{$0$}}\\
\multicolumn{1}{c}{}
\end{array}
\def\arraystretch{1.2}
\begin{array}{|cccccc}
\bullet&\gris&\gris&\gris&\gris\\
\bullet&\gris&\gris&\gris&\gris\\
\bullet&&\gris&\gris&\gris\\
\bullet&&\gris&\gris&\gris\\
\bullet&\bullet&&\gris&\gris\\
\bullet&\bullet&\bullet&\gris&\gris\\
\bullet&&\bullet&\bullet&\gris\\
\bullet&&&\bullet&\gris\\
\bullet&&&&\bullet\\
\hline
\multicolumn{1}{c}{}&\multicolumn{1}{c}{}&\multicolumn{1}{c}{}&\multicolumn{1}{c}{}&\multicolumn{1}{c}{\text{\tiny{$n$}}}
\end{array}
$$
where the bullets denote non-trivial elements. 
Consider the quasi-isomorphism of complexes $\rho:(E^{*,*}_2(X/Y),0)\to (E^{*,*}_1(X/Y),d_1)$ defined by taking sections of the projections
$\Ker(d^{p,q}_1)\twoheadrightarrow E_2^{p,q}(X/Y)$.
We next show that $\rho$ is multiplicative.
By bidegree reasons the non-trivial products in $E_2^{*,*}(X/Y)$ are
$E_2^{0,q}\times E_2^{0,q'}\lra E_2^{0,q+q'}$ for $q,q'\geq 0$
and
$E_2^{0,1}\times E_2^{p,n-p}\lra E_2^{p,n+1-p}$, for $0<p<n$.
Since $\rho$ is the identity on 
$E_2^{0,*}(X)\cong \Ker(d_1^{0,q})$,
it preserves the former products. It also preserves the unit $1\in E_2^{0,0}(X)$.
Arguing as in the proof of Theorem $\ref{linkconnectedformal}$, and using the description of $E_1^{*,*}(X/Y)$
given by Proposition $\ref{isolsingE1}$,
it is straightforward to see that the latter products are trivial and that the diagram
$$
\xymatrix{
\Ker(d_1^{0,1})\times E_2^{p,n-p}(X/Y)\ar[d]^{\rho\times\rho}\ar[r]&0\ar[d]\\
E_1^{0,1}(X/Y)\times E_1^{p,n-p}(X/Y)\ar[r]&E_1^{p,n+1-p}(X/Y)
}
$$
commutes.  
We may now proceed as in the proof of Theorem $\ref{linkconnectedformal}$. Hence the equivalent conditions $(b)$ and $(c)$ imply that $X/Y$ is formal.
Lastly, assume that $(d)$ is satisfied. The proof of Theorem $\ref{smoothdivisor}$ is valid for $X/Y$. Hence $X/Y$ is formal.
\end{proof}

\begin{example}
Let $Y\hookrightarrow X$ be a closed immersion of smooth projective varieties, with $X$ connected. Then 
$X/Y$ is a formal topological space.
\end{example}

\linespread{1}
\bibliographystyle{amsalpha}
\bibliography{bibliografia}
\mbox{}\\
\linespread{1.2}

\end{document}